\documentclass[12pt,oneside,reqno]{amsart}
\usepackage{array}
\usepackage{amssymb,tikz,amsmath,bm,color,graphicx,hyperref,anysize}
\usepackage{tikz-cd}
\usepackage{multicol}
\usepackage{mathtools}
\usepackage{subcaption}
\usetikzlibrary{circuits.logic.US,circuits.logic.IEC,fit}

\usetikzlibrary{cd}
\usepackage{enumitem}
\usepackage{newtxtext,newtxmath}
\usetikzlibrary{arrows.meta, decorations.markings}
\usepackage{tikz-3dplot}
\usepackage{tcolorbox}
\tcbuselibrary{skins, breakable}
\usetikzlibrary{arrows.meta,positioning}
\usepackage{blkarray}
\usetikzlibrary{arrows.meta}
\usepackage[numbers,sort&compress]{natbib}
\theoremstyle{plain}
\newtheorem{theorem}{Theorem}[section]
\newtheorem{lemma}[theorem]{Lemma}

\newenvironment{solution*}{%
  \par\noindent\textbf{Solution. }\normalfont
}{\qed}
\theoremstyle{definition}
\newtheorem{example}{Example}[section]
\newtheorem{definition}[theorem]{Definition}
\theoremstyle{remark}

\newtheorem{algorithm}{Algorithm}

\theoremstyle{nonitalic}
\newcommand{\St}{\operatorname{St}}
\newcommand{\Int}{\operatorname{Int}}

\usetikzlibrary{3d}
\usetikzlibrary{shapes.geometric, arrows}
\tikzstyle{startstop} = [rectangle, rounded corners, minimum width=3cm, minimum height=1cm,text centered, draw=black, fill=red!20]
\tikzstyle{process} = [rectangle, minimum width=3cm, minimum height=1cm, text centered, draw=black, fill=blue!10]
\tikzstyle{decision} = [diamond, minimum width=3cm, minimum height=1cm, text centered, draw=black, fill=green!20]
\tikzstyle{arrow} = [thick,->,>=stealth]
\usetikzlibrary{calc}

\makeatletter
\def\@setcopyright{}
\def\serieslogo@{}
\makeatother
\tikzset{
  mid arrow/.style={
    very thick,
    postaction={decorate},
    decoration={markings,
      mark=at position 0.5 with {\arrow{Stealth}}}
  },
  mid arrow rev/.style={
    very thick,
    postaction={decorate},
    decoration={markings,
      mark=at position 0.5 with {\arrow{Stealth[reversed]}}}
  }
}
\begin{document}
\marginsize{2.54 cm}{2.54 cm}{2.54 cm}{2.54 cm}
\author{Sanjay Mishra}
\address{Department of Mathematics \newline \indent Amity School of Applied Sciences, Amity University Lucknow Campus, UP India}
\email{drsmishraresearch@gmail.com}
\title[Subdivision of Simplicial Complex]{Subdivision of Simplicial Complex}
\begin{abstract}
This paper provides a self-contained exploration of subdivisions of simplicial complexes, with emphasis on barycentric subdivision. We present formal definitions of subdivisions, show how the realization of a complex is preserved under subdivision, and illustrate these concepts with explicit examples and detailed diagrams. The paper develops the general method of constructing subdivisions by starring from interior points, leading to the standard barycentric and derived subdivisions. We give precise statements and proofs demonstrating that repeated barycentric subdivision reduces the mesh below any prescribed scale, ensuring compatibility with given metrics and enabling applications such as simplicial approximation and homological analysis. Examples and TikZ illustrations clarify the structure of iterated subdivisions for finite complexes, highlighting their geometric and topological properties.
\end{abstract}
\maketitle
\section{Introduction}\label{s:Introduction}

Subdivision of simplicial complexes is a fundamental tool in topology, geometry, and combinatorics. By breaking down a given complex into finer complexes, subdivisions allow for greater flexibility in geometric and algebraic arguments, and they are essential in applications such as simplicial approximation, mesh refinement, and computational topology.

A central role is played by barycentric subdivision, which systematically splits each simplex of a complex into a collection of smaller simplices according to the combinatorics of its faces. This process preserves the geometric realization of the original complex while refining its simplicial structure. Iterated barycentric subdivision generates arbitrarily fine meshes compatible with any chosen metric, making it a cornerstone of both theoretical and computational approaches.

This paper develops the detailed theory of subdivisions—formal definitions, their topological properties, explicit construction procedures (including starring from interior points), and the pivotal fact that repeated subdivision makes all simplices uniformly small. Through explicit examples, diagrams, and rigorous statements, we show how barycentric subdivision is constructed, how it interacts with the geometry and topology of the original complex, and why it underlies major theorems such as simplicial approximation. The aim is to provide a thorough and concrete exposition, suitable for both learners new to subdivisions as well as researchers who need algorithmic and theoretical insight into the process. 

For convenience, throughout this paper we use the term “complex” in place of “simplicial complex” and adopt the following notational conventions. The Euclidean space is denoted by $\mathbb{R}^J$. The generalized Euclidean space, $\mathbb{E}^J$, is defined as the subset of $\mathbb{R}^J$ consisting of all points $(x_\alpha)_{\alpha \in J}$ such that $x_\alpha = 0$ for all but finitely many values of $\alpha$; it is a vector space called the generalized Euclidean space and is endowed with the metric $|x - y| = \max \{|x_{\alpha} - y_{\alpha}|\}_{\alpha \in J}$. A complex is denoted by $\mathcal{K}$, and its geometric realization is written as $|\mathcal{K}|$. The notation $\operatorname{Sd}(\mathcal{K})$ stands for a subdivision of the complex $\mathcal{K}$, while $\operatorname{Bsd}(\mathcal{K})$ refers specifically to its barycentric subdivision. The barycenter of a simplex $\sigma$ is denoted by $b_{\sigma}$.

\section{Subdivision of complex}
Subdivision of a simplicial complex is a process in which the complex is refined into smaller simplices while preserving its geometric realization. This technique allows for breaking down simplices of the original complex into collections of finer simplices that cover the same underlying space.

\begin{definition}[Subdivision of complex]\label{d:Subdivision of complex}
Let $\mathcal{K}$ be geometric complex in $\mathbb{E}^{J}$. A complex $\operatorname{Sd}(\mathcal{K})$ is called subdivision of $\mathcal{K}$ if: 
\begin{enumerate}
    \item For every simplex  $\sigma' \in \operatorname{Sd}(\mathcal{K})$ there exists a $\sigma \in \mathcal{K}$ with $\sigma' \subset \sigma$.
    \item For every simplex $\sigma \in \mathcal{K}$ there exist finitely many simplices $\sigma_{1}', \ldots, \sigma_{n}' \in \operatorname{Sd}(\mathcal{K})$ such that $\sigma = \cup_{i = 1}^{n} \sigma_{i}'$.
\end{enumerate}
\end{definition}
The second condition, together with the first condition of the above definition, leads to the following standard result, which describes how the realization of a complex is preserved under subdivision.

\begin{theorem}[Realization preserved under subdivision]\label{t:Realization preserved under subdivision}
Let $\mathcal{K}$ be a geometric simplicial complex in an ambient Euclidean space and let $\mathcal{L}=\operatorname{Sd}(\mathcal{K})$ be a subdivision of $\mathcal{K}$.
Then
\[
\cup_{\sigma'\in \mathcal{L}}\sigma' \;=\; \cup_{\sigma\in \mathcal{K}}\sigma.
\]
In particular, their geometric realizations determine the same subset of the ambient space, so
\[
|\mathcal{L}| = |\mathcal{K}|
\]
as topological spaces with the subspace topology; more generally, there is a canonical homeomorphism
\[
|\mathcal{L}|\cong |\mathcal{K}|.
\]
\end{theorem}

\begin{proof}
We prove the equality of unions by showing both inclusions, using only the defining properties of a subdivision.

\medskip
\noindent\emph{Inclusion $\displaystyle \cup_{\sigma'\in \mathcal{L}}\sigma' \subseteq \cup_{\sigma\in \mathcal{K}}\sigma$.}
Let $x \in \cup_{\sigma'\in \mathcal{L}}\sigma'$. Then $x \in \sigma'_0$ for some simplex $\sigma'_0 \in \mathcal{L}$.
By the refinement condition of a subdivision, there exists $\sigma_0 \in \mathcal{K}$ such that $\sigma'_0 \subseteq \sigma_0$.
Hence $x \in \sigma_0 \subseteq \cup_{\sigma\in \mathcal{K}}\sigma$, proving the first inclusion.

\noindent\emph{Inclusion $\displaystyle \cup_{\sigma\in \mathcal{K}}\sigma \subseteq \cup_{\sigma'\in \mathcal{L}}\sigma'$.}
Let $x \in \cup_{\sigma\in \mathcal{K}}\sigma$. Then $x \in \sigma_0$ for some simplex $\sigma_0 \in \mathcal{K}$.
By the subdivision covering condition, there exist finitely many simplices $\sigma'_1,\dots,\sigma'_n \in \mathcal{L}$ such that
\[
\sigma_0 \;=\; \cup_{i=1}^n \sigma'_i.
\]
Therefore $x \in \sigma'_i$ for some $i$, and hence $x \in \cup_{\sigma'\in \mathcal{L}}\sigma'$, proving the second inclusion.

\medskip
Combining the two inclusions yields
\[
\cup_{\sigma'\in \mathcal{L}}\sigma' \;=\; \cup_{\sigma\in \mathcal{K}}\sigma.
\]

\medskip
\noindent\emph{Equality/homeomorphism of realizations.}
When $|\mathcal{K}|$ and $|\mathcal{L}|$ are both regarded as subsets of the same ambient Euclidean space with the subspace topology, the preceding equality of unions shows that $|\mathcal{K}|=|\mathcal{L}|$ as topological subspaces. If, instead, one treats realizations abstractly (not as the identical subset), the identity on the common underlying subset induces a canonical bijection $|\mathcal{L}|\to|\mathcal{K}|$ which is a homeomorphism since both carry the subspace topology from the same ambient space. Thus there is a canonical homeomorphism
\[
|\mathcal{L}| \;\cong\; |\mathcal{K}|.
\]
\end{proof}

For better understanding let us look at an example that shows how the realization is preserved during the subdivision of the complex.

\begin{example}
Let $\mathcal{K}$ be the geometric simplicial complex in $\mathbb{R}^2$ consisting of a single 2–simplex $\triangle ABC$ together with all its faces as
\[\mathcal{K} = \{A, B, C, [AB], [BC], [CA], [ABC]\}\]

Choose a mid point $D$ on the edge $[A,B]$. Define $\mathcal{L}$ to be the simplicial complex whose 2–simplices are $\triangle ADC$ and $\triangle BDC$, together with all their faces as
\[\mathcal{L} = \{A, B, C, D, [AD], [DB], [BC], [CA],  [ADC], [BDC]\}\]

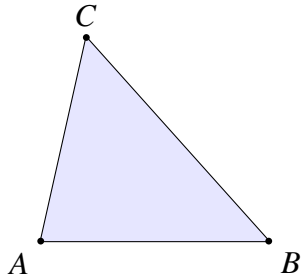
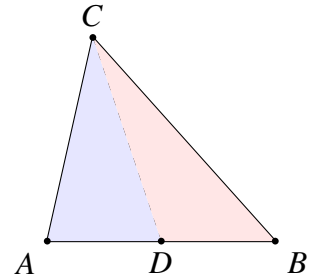
\begin{figure}[h!]
  \centering
  \begin{subfigure}[t]{0.45\textwidth}
    \centering
    \begin{tikzpicture}[scale=3]
      \coordinate (A) at (0,0);
      \coordinate (B) at (1,0);
      \coordinate (C) at (0.2,0.9);

      \draw[thick] (A)--(B)--(C)--cycle;


      \fill[blue!10] (A)--(B)--(C)--cycle;

      \fill (A) circle (0.015);
      \fill (B) circle (0.015);
      \fill (C) circle (0.015);

      \node[below left]  at (A) {$A$};
      \node[below right] at (B) {$B$};
      \node[above]       at (C) {$C$};
    \end{tikzpicture}
    \caption{Complex $\mathcal{K}$}
    \label{fig:Complex K$}
  \end{subfigure}
  \hfill
  \begin{subfigure}[t]{0.45\textwidth}
    \centering
    \begin{tikzpicture}[scale=3]
      \coordinate (A) at (0,0);
      \coordinate (B) at (1,0);
      \coordinate (C) at (0.2,0.9);
      \coordinate (D) at ($(A)!0.5!(B)$);

      \draw[thick] (A)--(B)--(C)--cycle;

      \draw[thick, dashed] (D)--(C);

      \fill[blue!10] (A)--(D)--(C)--cycle;
      \fill[red!10]  (B)--(D)--(C)--cycle;

      \fill (A) circle (0.015);
      \fill (B) circle (0.015);
      \fill (C) circle (0.015);
      \fill (D) circle (0.015);

      \node[below left]  at (A) {$A$};
      \node[below right] at (B) {$B$};
      \node[above]       at (C) {$C$};
      \node[below]       at (D) {$D$};
    \end{tikzpicture}
    \caption{Including mid point $D$ of edge $[AB]$ for sub-division complex $\mathcal{L}$ of the complex $\mathcal{K}$}
    \label{fig:Including mid point D of edge AB for sub-division complex L of the complex K$}
  \end{subfigure}

  \caption{Sub-division of complex $\mathcal{K}$  into the complex $\mathcal{L}$}
  \label{fig:Sub-division of complex K  into the complex L}
\end{figure}

See in the figure \eqref{fig:Sub-division of complex K  into the complex L}. 
Let us examine how this realization is preserved during the subdivision of the complex.

We claim $\cup_{\sigma'\in \mathcal{L}} \sigma' \;=\; \cup_{\sigma\in \mathcal{K}} \sigma \;=\; \triangle ABC$. First, every simplex of $\mathcal{L}$ lies in a simplex of $\mathcal{K}$:
indeed, $\triangle ADC \subset \triangle ABC$, $\triangle BDC \subset \triangle ABC$, and each edge/vertex of $\mathcal{L}$ is contained in an edge/vertex of $\mathcal{K}$.
Hence
\[
\cup_{\sigma'\in \mathcal{L}} \sigma' \;\subseteq\; \cup_{\sigma\in \mathcal{K}} \sigma \;=\; \triangle ABC.
\]

Conversely, $\triangle ABC$ decomposes as the union of the two triangles $\triangle ADC$ and $\triangle BDC$, whose interiors are disjoint and which meet along the segment $[D,C]$.
Thus
\[
\triangle ABC \;=\; \triangle ADC \,\cup\, \triangle BDC \;\subseteq\; \cup_{\sigma'\in \mathcal{L}} \sigma',
\]
so $\cup_{\sigma\in \mathcal{K}} \sigma \subseteq \cup_{\sigma'\in \mathcal{L}} \sigma'$.
Therefore the unions coincide $\cup_{\sigma'\in \mathcal{L}} \sigma' \;=\; \cup_{\sigma\in \mathcal{K}} \sigma \;=\; \triangle ABC$.

Consequence for realizations.
Viewing both $|\mathcal{K}|$ and $|\mathcal{L}|$ as subspaces of $\mathbb{R}^2$ with the subspace topology, the above equality of unions implies $|\mathcal{L}| \;=\; |\mathcal{K}|$. Equivalently, there is a canonical homeomorphism $|\mathcal{L}| \cong |\mathcal{K}|$ induced by the identity on $\triangle ABC$.
\end{example}

\begin{theorem}[Transitivity of subdivision]\label{t:Transitivity of subdivision}
If $\mathcal{K}''$ is a sub-division of $\mathcal{K}'$ and if $\mathcal{K}'$ is a sub-division of  $\mathcal{K}$, then $\mathcal{K}''$ is a sub-division of $\mathcal{K}$.
\end{theorem}

\begin{proof}
If $\mathcal{K}''$ is a sub-division of $\mathcal{K}'$ and if $\mathcal{K}'$ is a sub-division of  $\mathcal{K}$, then $\mathcal{K}''$ is a sub-division of $\mathcal{K}$ can be by directly by the definition \eqref{d:Subdivision of complex}.

Let $\tau'' \in \mathcal{K}''$. Since $\mathcal{K}''$ is a subdivision of $\mathcal{K}'$, there exists $\tau' \in \mathcal{K}'$ with $\tau'' \subseteq \tau'$.
Since $\mathcal{K}'$ is a subdivision of $\mathcal{K}$, there exists $\sigma \in \mathcal{K}$ with $\tau' \subseteq \sigma$.
Thus $\tau'' \subseteq \sigma$, establishing first condition for $\mathcal{K}''$ relative to $\mathcal{K}$.

Let $\sigma \in \mathcal{K}$.
Because $\mathcal{K}'$ is a subdivision of $\mathcal{K}$, there exist finitely many simplices $\tau'_1,\dots,\tau'_r \in \mathcal{K}'$ such that
\[
\sigma \;=\; \bigcup_{j=1}^{r} \tau'_j.
\]
For each $j$, since $\mathcal{K}''$ is a subdivision of $\mathcal{K}'$, there exist finitely many simplices $\tau''_{j,1},\dots,\tau''_{j,s_j} \in \mathcal{K}''$ with
\[
\tau'_j \;=\; \bigcup_{i=1}^{s_j} \tau''_{j,i}.
\]
Therefore
\[
\sigma \;=\; \bigcup_{j=1}^{r} \tau'_j
\;=\; \bigcup_{j=1}^{r} \ \bigcup_{i=1}^{s_j} \tau''_{j,i},
\]
a finite union of simplices in $\mathcal{K}''$.
Hence second condition holds for $\mathcal{K}''$ relative to $\mathcal{K}$.

Since both conditions are satisfied, therefore $\mathcal{K}''$ is a subdivision of $\mathcal{K}$.
\end{proof}

\begin{definition}[Induced sub-division of a sub-complex]\label{d:[Induced sub-division of a sub-complex}
Let $\mathcal{K}'$ be a subdivision of a geometric simplicial complex $\mathcal{K}$, and let $\mathcal{K}_0$ be a subcomplex of $\mathcal{K}$.
The complex
\[
\mathcal{K}'\big|_{\mathcal{K}_0}\;=\;\bigl\{\;\sigma'\in \mathcal{K}' \;:\; \sigma'\subseteq \,|\mathcal{K}_0| \;\bigr\}.
\]
is a sub-division of $\mathcal{K}_0$, called the sub-division of $\mathcal{K}_0$ induced by $\mathcal{K}'$.
\end{definition}

\begin{example}[Induced sub-division on a sub-complex]\label{eg:Induced sub-division on a sub-complex}
Let $\mathcal{K}$ be the geometric simplicial complex in $\mathbb{R}^2$ consisting of a single triangle $\triangle ABC$ together with all its faces.

\[\mathcal{K} = \{A, B, C, [AB], [BC], [CA], [ABC]\}\]

Form a subdivision $\mathcal{K}'$ by choosing a point $D$ on the edge $[A,B]$ and replacing the facet $\triangle ABC$ by the two triangles $\triangle ADC$ and $\triangle BDC$, including all their faces.

\[\mathcal{K}' = \{A, B, C, D, [AD], [DB], [BC], [CA],  [ADC], [BDC]\}\]

Let $\mathcal{K}_0\subset \mathcal{K}$ be the edge subcomplex generated by $[A,B]$, i.e., $\mathcal{K}_0=\{A, B, [AB]\}$.

The induced subdivision of $\mathcal{K}_0$ from $\mathcal{K}'$ is
\[
\mathcal{K}'\big|_{\mathcal{K}_0}
\;=\;
\bigl\{\,\sigma'\in \mathcal{K}' \;:\; \sigma' \subseteq |\mathcal{K}_0| \,\bigr\}.
\]
Since $|\mathcal{K}_0|=[AB]$, the simplices of $\mathcal{K}'$ lying in $[AB]$ are precisely the edges $[AD]$, $[DB]$ and the vertices $A,B,D$.
Hence
\[
\mathcal{K}'\big|_{\mathcal{K}_0}
\;=\;
\bigl\{A, B, D, [AD], [BD]\bigr\},
\]
which is a subdivision of the edge $[AB]$ obtained by inserting the vertex $D$.
This induced complex refines $\mathcal{K}_0$ simplexwise, and each simplex of $\mathcal{K}_0$ is a finite union of simplices from $\mathcal{K}'\big|_{\mathcal{K}_0}$ (notably, $[AB]=[AD]\cup[DB]$), so $\mathcal{K}'\big|_{\mathcal{K}_0}$ is a subdivision of $\mathcal{K}_0$. See the figure \eqref{fig:Induced su-division on a sub-complex}.
\begin{figure}[h!]
  \centering
  \begin{subfigure}[t]{0.45\textwidth}
    \centering
    \begin{tikzpicture}[scale=3]
      \coordinate (A) at (0,0);
      \coordinate (B) at (1,0);
      \coordinate (C) at (0.2,0.9);

      \draw[thick] (A)--(B)--(C)--cycle;


      \fill[blue!10] (A)--(B)--(C)--cycle;

      \fill (A) circle (0.015);
      \fill (B) circle (0.015);
      \fill (C) circle (0.015);

      \node[below left]  at (A) {$A$};
      \node[below right] at (B) {$B$};
      \node[above]       at (C) {$C$};
    \end{tikzpicture}
    \caption{Complex $\mathcal{K}$}
    \label{fig:Complex K$}
  \end{subfigure}
\begin{subfigure}[t]{0.45\textwidth}
    \centering
    \begin{tikzpicture}[scale=3]
      \coordinate (A) at (0,0);
      \coordinate (B) at (1,0);

      \draw[thick] (A)--(B)--cycle;


      \fill[blue!10] (A)--(B)--cycle;

      \fill (A) circle (0.015);
      \fill (B) circle (0.015);

      \node[below left]  at (A) {$A$};
      \node[below right] at (B) {$B$};
    \end{tikzpicture}
    \caption{Sub-complex $\mathcal{K}_{0} \subset{\mathcal{K}}$}
    \label{fig:Complex K0}
  \end{subfigure}
 
  \hfill
  \begin{subfigure}[t]{0.45\textwidth}
    \centering
    \begin{tikzpicture}[scale=3]
      \coordinate (A) at (0,0);
      \coordinate (B) at (1,0);
      \coordinate (C) at (0.2,0.9);
      \coordinate (D) at ($(A)!0.5!(B)$);

      \draw[thick] (A)--(B)--(C)--cycle;

      \draw[thick, dashed] (D)--(C);

      \fill[blue!10] (A)--(D)--(C)--cycle;
      \fill[red!10]  (B)--(D)--(C)--cycle;

      \fill (A) circle (0.015);
      \fill (B) circle (0.015);
      \fill (C) circle (0.015);
      \fill (D) circle (0.015);

      \node[below left]  at (A) {$A$};
      \node[below right] at (B) {$B$};
      \node[above]       at (C) {$C$};
      \node[below]       at (D) {$D$};
    \end{tikzpicture}
    \caption{Sub-division $\mathcal{K}'$ of $\mathcal{K}$}
    \label{fig:Including mid point D of edge AB for sub-division complex L of the complex K$}
  \end{subfigure}
\hfill
  \begin{subfigure}[t]{0.45\textwidth}
    \centering
    \begin{tikzpicture}[scale=3]
      \coordinate (A) at (0,0);
      \coordinate (B) at (1,0);
      \coordinate (D) at ($(A)!0.5!(B)$);
\draw[blue, thick] (A) -- (D);
\draw[red, thick] (D) -- (B);
      \fill (A) circle (0.015);
      \fill (B) circle (0.015);
      \fill (D) circle (0.015);

      \node[below left]  at (A) {$A$};
      \node[below right] at (B) {$B$};
     \node[below]       at (D) {$D$};
    \end{tikzpicture}
    \caption{Sub-division $\mathcal{K}'\big|_{\mathcal{K}_0}$ of $\mathcal{K}_{0}$ induced by $\mathcal{K}'$}
    \label{fig:Including mid point D of edge AB for sub-division complex L of the complex K$}
  \end{subfigure}
  \caption{Induced sub-division $\mathcal{K}'\big|_{\mathcal{K}_0}$ of sub-complex $\mathcal{K}_{0}$ of complex $\mathcal{K}$ induced by sub-division $\mathcal{K}'$ of $\mathcal{K}$}
  \label{fig:Induced su-division on a sub-complex}
\end{figure}
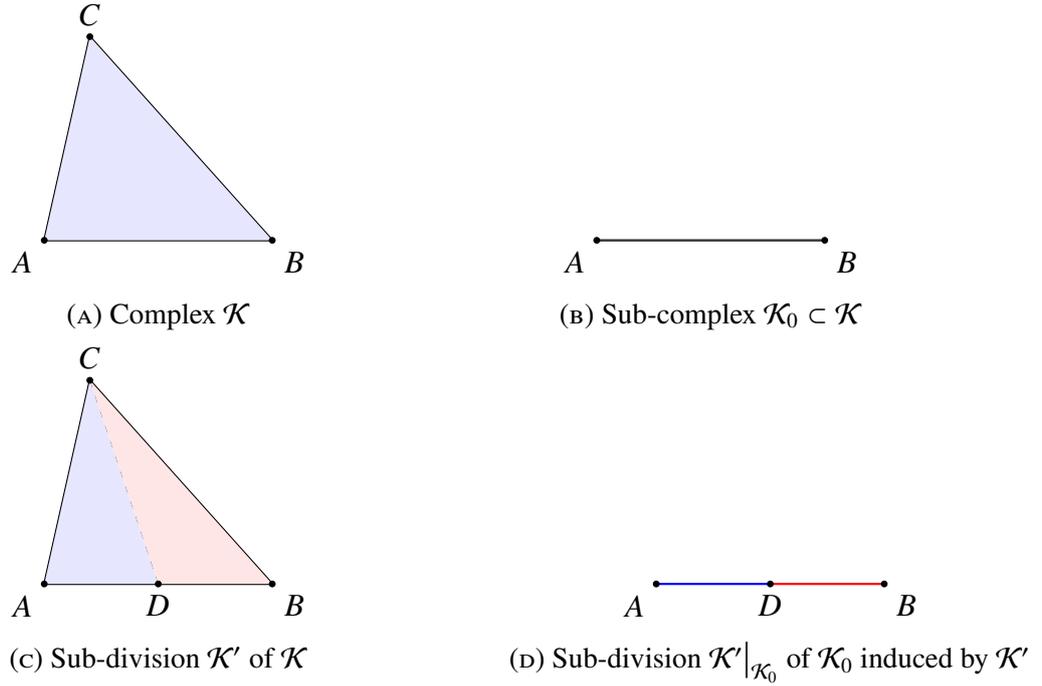
\end{example}

The next lemma captures a local containment property of stars under subdivision: every new vertex in a subdivided complex has its open star contained in the open star of a suitable original vertex, namely any vertex of the simplex that “carries” the new vertex. Moreover, if $\sigma$ is the simplex of $\mathcal{K}$ with $w \in \operatorname{Int}\sigma$ (the carrier of $w$), then this inclusion holds precisely for those $v$ that are vertices of $\sigma$.

\begin{definition}[Open star of a vertex of complex]\label{d:Open star of a vertex of complex}
Let $\mathcal{K}$ be a geometric simplicial complex and let $v$ be a vertex of $\mathcal{K}$.
The open star of $v$ in $\mathcal{K}$ is
\[
\St(v,\mathcal{K})
\;=\;
\bigcup \bigl\{\, \Int(\sigma) \;:\; \sigma \in \mathcal{K},\ v \in \sigma \,\bigr\},
\]
i.e., the union of the relative interiors of all simplices of $\mathcal{K}$ that contain $v$.
\end{definition}


\begin{example}[Open star in a triangle]\label{eg:Open star in a triangle}
Let $\mathcal{K}$ be the geometric simplicial complex in $\mathbb{R}^2$ consisting of the triangle $[ABC]$ together with all its faces:
\[
\mathcal{K}=\bigl\{[ABC],\ [AB],\ [BC],\ [CA],\ A,\ B,\ C \bigr\}.
\]
For the vertex $A$, the open star $\St(A,\mathcal{K})$ is the union of the relative interiors of all simplices of $\mathcal{K}$ that contain $A$, namely
\[
\St(A,\mathcal{K})
\;=\;
\operatorname{Int}[ABC]\ \cup\ \operatorname{Int}[AB]\ \cup\ \operatorname{Int}[AC]\ \cup\ \{A\}.
\]
Similarly,
\begin{gather*}
\St(B,\mathcal{K})=\operatorname{Int}[ABC]\cup \operatorname{Int}[AB]\cup \operatorname{Int}[BC]\cup \{B\},\\ 
\St(C,\mathcal{K})=\operatorname{Int}[ABC]\cup \operatorname{Int}[BC]\cup \operatorname{Int}[CA]\cup \{C\}.   \end{gather*}

In words, the open star of a vertex in a single 2–simplex is the interior of the triangle together with the interiors of the two edges incident to that vertex and the vertex itself. See the figure \eqref{fig:stars-ABC}.

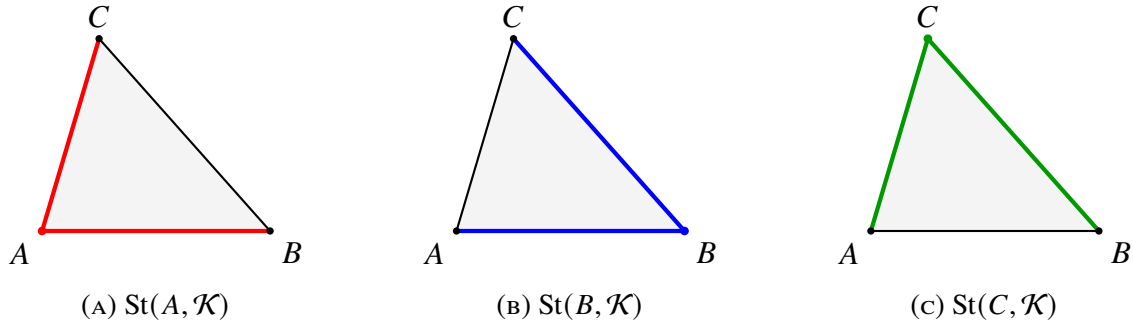
\begin{figure}[h!]
\centering

\begin{subfigure}[t]{0.32\textwidth}
\centering
\begin{tikzpicture}[scale=3]
  \coordinate (A) at (0,0);
  \coordinate (B) at (1,0);
  \coordinate (C) at (0.25,0.85);

  \fill[gray!9] (A)--(B)--(C)--cycle;
  \draw[thick] (A)--(B)--(C)--cycle;

  \draw[line width=1.6pt, red] (A)--(B);
  \draw[line width=1.6pt, red] (A)--(C);

  \fill[red] (A) circle (0.019);
  \fill (B) circle (0.016);
  \fill (C) circle (0.016);

  \node[below left]  at (A) {$A$};
  \node[below right] at (B) {$B$};
  \node[above]       at (C) {$C$};
\end{tikzpicture}
\caption{$\St(A,\mathcal{K})$}
\label{fig:star-A}
\end{subfigure}
\hfill
\begin{subfigure}[t]{0.32\textwidth}
\centering
\begin{tikzpicture}[scale=3]
  \coordinate (A) at (0,0);
  \coordinate (B) at (1,0);
  \coordinate (C) at (0.25,0.85);

  \fill[gray!9] (A)--(B)--(C)--cycle;
  \draw[thick] (A)--(B)--(C)--cycle;

  \draw[line width=1.6pt, blue] (B)--(A);
  \draw[line width=1.6pt, blue] (B)--(C);

  \fill (A) circle (0.016);
  \fill[blue] (B) circle (0.019);
  \fill (C) circle (0.016);

  \node[below left]  at (A) {$A$};
  \node[below right] at (B) {$B$};
  \node[above]       at (C) {$C$};
\end{tikzpicture}
\caption{$\St(B,\mathcal{K})$}
\label{fig:star-B}
\end{subfigure}
\hfill
\begin{subfigure}[t]{0.32\textwidth}
\centering
\begin{tikzpicture}[scale=3]
  \coordinate (A) at (0,0);
  \coordinate (B) at (1,0);
  \coordinate (C) at (0.25,0.85);

  \fill[gray!9] (A)--(B)--(C)--cycle;
  \draw[thick] (A)--(B)--(C)--cycle;

  \draw[line width=1.6pt, green!60!black] (C)--(A);
  \draw[line width=1.6pt, green!60!black] (C)--(B);

  \fill (A) circle (0.016);
  \fill (B) circle (0.016);
  \fill[green!60!black] (C) circle (0.019);

  \node[below left]  at (A) {$A$};
  \node[below right] at (B) {$B$};
  \node[above]       at (C) {$C$};
\end{tikzpicture}
\caption{$\St(C,\mathcal{K})$}
\label{fig:star-C}
\end{subfigure}

\caption{Open stars for vertices in the triangle complex $\mathcal{K}$}
\label{fig:stars-ABC}
\end{figure}

\end{example}

\begin{lemma}[Star inclusion under subdivision]\label{l:Star inclusion under subdivision}
Let $\mathcal{K}'$ be subdivision of $\mathcal{K}$. Then for each vertex $w$ of $\mathcal{K}'$, there is a vertex $v$ of $\mathcal{K}$ such that $\operatorname{St}(w, \mathcal{K}') \subset \operatorname{St}(v, \mathcal{K})$. Indeed, if $\sigma$ is the simplex of $\mathcal{K}$ such that $w \in \operatorname{Int} \sigma$, then this inclusion holds precisely when $v$ is a vertex of $\sigma$. 
\end{lemma}

\begin{proof}
Given a geometric simplicial complex $\mathcal{K}$ and a subdivision $\mathcal{K}'$ of $\mathcal{K}$, together with an arbitrary vertex $w$ of $\mathcal{K}'$ and its carrier simplex $\sigma$ in $\mathcal{K}$ characterized by $w\in \operatorname{Int}\sigma$. To prove, there exists a vertex $v$ of $\mathcal{K}$ such that $\operatorname{St}(w,\mathcal{K}')\subset \operatorname{St}(v,\mathcal{K})$; moreover, the full set of such vertices is $\mathrm{Vert}(\sigma)$, i.e., the inclusion holds precisely for those $v$ that are vertices of the carrier simplex $\sigma$.

Let $\mathcal{K}'$ be a subdivision of $\mathcal{K}$ and let $w$ be a vertex of $\mathcal{K}'$.
By the defining property of a subdivision, there exists a unique simplex $\sigma \in \mathcal{K}$ such that $w \in \operatorname{Int}\sigma$; this $\sigma$ is called the carrier of $w$ in $\mathcal{K}$.
Write $\sigma=[v_0,\dots,v_m]$ with vertices $v_i\in \mathrm{Vert}(\mathcal{K})$.

\begin{description}
 \item[Step 1: Star containment for any vertex of the carrier]
Fix a vertex $v$ of $\sigma$.
We claim $\operatorname{St}(w,\mathcal{K}') \subset \operatorname{St}(v,\mathcal{K})$.
Let $x \in \operatorname{St}(w,\mathcal{K}')$.
By definition of open star, there exists a simplex $\tau' \in \mathcal{K}'$ with $w \in \operatorname{Int}\tau'$ and $x \in \operatorname{Int}\tau'$.
Since $\mathcal{K}'$ subdivides $\mathcal{K}$, there exists $\tau \in \mathcal{K}$ such that $\tau'\subseteq \tau$.
Moreover, $w \in \operatorname{Int}\tau' \subseteq \tau$ implies $w \in \operatorname{Int}\tau$ because interiors are preserved under refinement at interior points.
By uniqueness of the carrier simplex, $\tau=\sigma$.
Hence $\tau' \subseteq \sigma$.
Because $v$ is a vertex of $\sigma$, every point of $\operatorname{Int}\tau'$ lies in a simplex of $\mathcal{K}$ that contains $v$ (indeed, $\sigma$ itself contains both $v$ and $\tau'$).
Thus $x \in \operatorname{Int}\tau' \subseteq \sigma \subset \operatorname{St}(v,\mathcal{K})$.
Since $x$ was arbitrary, $\operatorname{St}(w,\mathcal{K}') \subset \operatorname{St}(v,\mathcal{K})$.

\item[Step 2: Characterization of all such vertices]
Suppose $v\in \mathrm{Vert}(\mathcal{K})$ satisfies $\operatorname{St}(w,\mathcal{K}') \subset \operatorname{St}(v,\mathcal{K})$.
We show that $v$ must be a vertex of the carrier $\sigma$.
Pick a point $x \in \operatorname{Int}\sigma$ sufficiently close to $w$ so that $x \in \operatorname{St}(w,\mathcal{K}')$ (this is possible because $w\in \operatorname{Int}\sigma$ and $\sigma$ is the unique carrier).
By the assumed inclusion, $x \in \operatorname{St}(v,\mathcal{K})$.
Hence there exists a simplex $\rho \in \mathcal{K}$ containing $v$ with $x \in \operatorname{Int}\rho$.
But $x \in \operatorname{Int}\sigma \cap \operatorname{Int}\rho$ implies $\rho=\sigma$ by uniqueness of the simplex whose interior contains $x$.
Therefore $v$ is a vertex of $\sigma$.
\end{description}

Combining Steps 1 and 2, we conclude:
for every vertex $w$ of $\mathcal{K}'$ with carrier $\sigma\in \mathcal{K}$, the inclusion $\operatorname{St}(w,\mathcal{K}') \subset \operatorname{St}(v,\mathcal{K})$ holds for each vertex $v$ of $\sigma$, and it holds only for those $v$.
This proves the lemma.
\end{proof}

\begin{example}[Star containment under a triangle subdivision]
Let $\mathcal{K}$ be the geometric simplicial complex in $\mathbb{R}^2$ consisting of a single 2-simplex together with all its faces,
\[
\mathcal{K}=\{A, B, C,\ [AB], [BC], [CA],\ [ABC]\}.
\]
Let $\sigma=[ABC]$.
Form a subdivision $\mathcal{K}'$ (See in figure \eqref{fig:Complex and its sub-division}) by choosing a point $D\in \operatorname{Int}\sigma$ (e.g., the barycenter) and coning to the vertices:
\[
\mathcal{K}'
=
\Bigl\{
A,B,C,D,\ 
[AB], [BC], [CA], [AD], [BD], [CD],\ 
[DAB], [DBC], [DCA]
\Bigr\}.
\]

\begin{figure}[h!]
\centering

\begin{subfigure}[t]{0.47\textwidth}
\centering
\begin{tikzpicture}[scale=3]

  \coordinate (A) at (0,0);
  \coordinate (B) at (1,0);
  \coordinate (C) at (0.25,0.85);

  \fill[gray!10] (A)--(B)--(C)--cycle;

  \draw[thick] (A)--(B);
  \draw[thick] (B)--(C);
  \draw[thick] (C)--(A);

  \fill (A) circle (0.016);
  \fill (B) circle (0.016);
  \fill (C) circle (0.016);

  \node[below left]  at (A) {$A$};
  \node[below right] at (B) {$B$};
  \node[above]       at (C) {$C$};

\end{tikzpicture}
\caption{Original complex $\mathcal{K}$.}
\label{fig:K-original}
\end{subfigure}
\hfill
\begin{subfigure}[t]{0.47\textwidth}
\centering
\begin{tikzpicture}[scale=3]

  \coordinate (A) at (0,0);
  \coordinate (B) at (1,0);
  \coordinate (C) at (0.25,0.85);

  \path let \p1=(A), \p2=(B), \p3=(C)
  in coordinate (D) at ({(\x1+\x2+\x3)/3},{(\y1+\y2+\y3)/3});

  \fill[blue!10]  (D)--(A)--(B)--cycle;
  \fill[red!10]   (D)--(B)--(C)--cycle;
  \fill[green!10] (D)--(C)--(A)--cycle;

  \draw[thick] (A)--(B)--(C)--cycle;

  \draw[thick, dashed] (A)--(D);
  \draw[thick, dashed] (B)--(D);
  \draw[thick, dashed] (C)--(D);

  \fill (A) circle (0.016);
  \fill (B) circle (0.016);
  \fill (C) circle (0.016);
  \fill (D) circle (0.016);

  \node[below left]  at (A) {$A$};
  \node[below right] at (B) {$B$};
  \node[above]       at (C) {$C$};
  \node[right]       at (D) {$D$};

\end{tikzpicture}
\caption{Subdivision $\mathcal{K}'$ of $\mathcal{K}$.}
\label{fig:K-subdivision}
\end{subfigure}

\caption{Complex $\mathcal{K}$ and its sub-division $\mathcal{K}'$.}
\label{fig:Complex and its sub-division}
\end{figure}
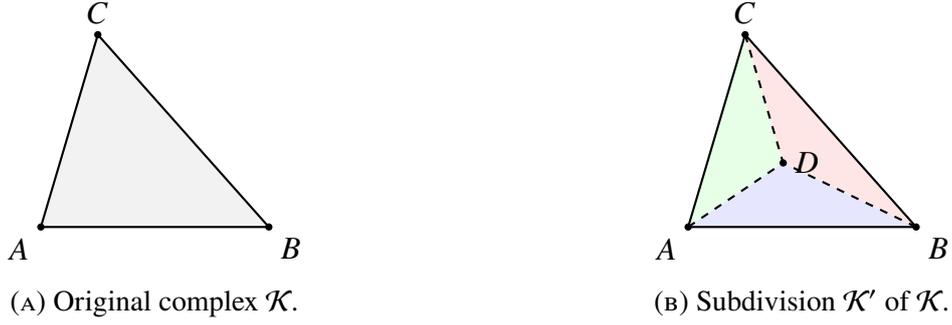
\paragraph{Open stars in $\mathcal{K}'$.}
\[
\begin{aligned}
\operatorname{St}(A,\mathcal{K}') &=
\operatorname{Int}[DAB]\ \cup\ \operatorname{Int}[DCA]\ \cup\
\operatorname{Int}[AB]\ \cup\ \operatorname{Int}[AC]\ \cup\ \operatorname{Int}[AD]\ \cup\ \{A\},\\
\operatorname{St}(B,\mathcal{K}') &=
\operatorname{Int}[DAB]\ \cup\ \operatorname{Int}[DBC]\ \cup\
\operatorname{Int}[AB]\ \cup\ \operatorname{Int}[BC]\ \cup\ \operatorname{Int}[BD]\ \cup\ \{B\},\\
\operatorname{St}(C,\mathcal{K}') &=
\operatorname{Int}[DCA]\ \cup\ \operatorname{Int}[DBC]\ \cup\
\operatorname{Int}[AC]\ \cup\ \operatorname{Int}[BC]\ \cup\ \operatorname{Int}[CD]\ \cup\ \{C\},\\
\operatorname{St}(D,\mathcal{K}') &=
\operatorname{Int}[DAB]\ \cup\ \operatorname{Int}[DBC]\ \cup\ \operatorname{Int}[DCA]\ \cup\
\operatorname{Int}[AD]\ \cup\ \operatorname{Int}[BD]\ \cup\ \operatorname{Int}[CD]\ \cup\ \{D\}.
\end{aligned}
\]

\paragraph{Claim.}
For the new vertex $w:=D\in \mathrm{Vert}(\mathcal{K}')$,
\[
\operatorname{St}(D,\mathcal{K}') \subset \operatorname{St}(A,\mathcal{K}),\qquad
\operatorname{St}(D,\mathcal{K}') \subset \operatorname{St}(B,\mathcal{K}),\qquad
\operatorname{St}(D,\mathcal{K}') \subset \operatorname{St}(C,\mathcal{K}),
\]
and these are exactly the vertices $v$ of $\mathcal{K}$ for which the inclusion holds (namely $v\in\{A,B,C\}$, the vertices of the carrier simplex $\sigma$).

\paragraph{Verification.}
Since $D\in \operatorname{Int}\sigma$ and $\sigma=[ABC]$ is the unique simplex of $\mathcal{K}$ whose interior contains $D$, the carrier of $D$ is $\sigma$.
In $\mathcal{K}'$, the open star $\operatorname{St}(D,\mathcal{K}')$ is the union of the interiors of the three triangles $[DAB]$, $[DBC]$, $[DCA]$ and their incident open faces, each lying inside $\sigma$.
Because $A$ is a vertex of $\sigma$, $\operatorname{St}(A,\mathcal{K})$ is the union of the interiors of all simplices of $\mathcal{K}$ containing $A$, namely the interior of $\sigma$, the edges $[AB]$, $[AC]$, and the vertex $A$.
Thus $\operatorname{St}(D,\mathcal{K}')\subset \operatorname{St}(A,\mathcal{K})$, and by symmetry the same holds for $B$ and $C$.
Conversely, if $v$ is a vertex of $\mathcal{K}$ with $\operatorname{St}(D,\mathcal{K}')\subset \operatorname{St}(v,\mathcal{K})$, then any point $x\in \operatorname{Int}\sigma$ sufficiently close to $D$ lies in $\operatorname{St}(D,\mathcal{K}')$ and hence in $\operatorname{St}(v,\mathcal{K})$.
Therefore $x\in \operatorname{Int}\rho$ for some simplex $\rho\in \mathcal{K}$ containing $v$; since $x\in \operatorname{Int}\sigma$, uniqueness of the containing simplex with nonempty interior gives $\rho=\sigma$, so $v$ is a vertex of $\sigma$.
\end{example}

The above lemma and example tell us a method to construct a subdivision of a complex by ``staring at the complex from an interior point". To this end, we are going to generalize this method to subdivision of complexes. 

\begin{lemma}\label{l:15.2}
If $\mathcal{K}$ is a complex, then intersection of any collection of sub-complexes of $\mathcal{K}$ is a sub-complex of $\mathcal{K}$. Conversely, if $\{\mathcal{K}_{\alpha}\}$ is a collection of complxes in $\mathbb{E}^{J}$, and if the intersection of every pair $|\mathcal{K}_{\alpha}| \cap |\mathcal{K}_{\beta}|$ is the polytype of a complex that is a sub-complex of both $\mathcal{K}_{\alpha}$ and $\mathcal{K}_{\beta}$, then the union $\cup \mathcal{K}_{\alpha}$ is a complex.
\end{lemma}

\begin{algorithm}[General method for subdivision of a complex by starring interior points, built skeleton-by-skeleton]\label{al:General method for subdivision of a complex by starring interior points, built skeleton-by-skeleton]}
Our objective is to construct a subdivision of a complex $\mathcal{K}$ inductively over skeleta by first subdividing the $p$-skeleton and then extending into each $(p+1)$-simplex by starring from an interior point. Follow these steps: 

\begin{description}
\item[Step 1: Data and notation] Let $\mathcal{K}$ be a simplicial complex.
For $p\ge 0$, denote by $\mathcal{K}^{(p)}$ the $p$-skeleton of $\mathcal{K}$, i.e., the subcomplex consisting of all simplices of dimension $\le p$.
Suppose $\mathcal{L}_p$ is a subdivision of $\mathcal{K}^{(p)}$ (same underlying space, refined simplices).

\item[Step 2: Boundary restriction on a $(p+1)$-simplex] Let $\sigma$ be a $(p+1)$-simplex of $\mathcal{K}$.
Its boundary $\operatorname{Bd}(\sigma)$ is the union of all $p$-faces of $\sigma$, hence a subcomplex of $\mathcal{K}^{(p)}$.
Since $\mathcal{L}_p$ subdivides $\mathcal{K}^{(p)}$, the boundary sphere $|\operatorname{Bd}(\sigma)|$ is also the underlying space of a subcomplex of $\mathcal{L}_p$.
Define
\[
\mathcal{L}_\sigma \;:=\; \{\, \tau \in \mathcal{L}_p \;:\; |\tau| \subset |\operatorname{Bd}(\sigma)| \,\}.
\]
Intuitively, $\mathcal{L}_\sigma$ is the restriction of the $p$-skeleton subdivision to the boundary of $\sigma$.

\item[Step 3: Choose an interior apex and cone]
Pick a point $w_\sigma \in \operatorname{Int}(\sigma)$ (e.g., the barycenter). Form the cone
\[
w_\sigma \ast \mathcal{L}_\sigma \;:=\; \bigl\{\, [w_\sigma \cup \tau] \text{ and all faces } : \tau \in \mathcal{L}_\sigma \,\bigr\}.
\]
Geometrically, this fills $\sigma$ by joining each face of the boundary subdivision to the interior apex $w_\sigma$, so that the underlying space of $w_\sigma \ast \mathcal{L}_\sigma$ is exactly $|\sigma|$.

\item[Step 4: Extend from the $p$-skeleton to the $(p+1)$-skeleton] Define
\[
\mathcal{L}_{p+1} \;:=\; \mathcal{L}_p \;\cup\; \bigcup_{\substack{\sigma \in \mathcal{K}\\ \dim \sigma = p+1}} \bigl( w_\sigma \ast \mathcal{L}_\sigma \bigr).
\]
Thus $\mathcal{L}_{p+1}$ agrees with $\mathcal{L}_p$ on $\mathcal{K}^{(p)}$ and, for each $(p+1)$-simplex $\sigma$, replaces the interior of $\sigma$ by the coned subdivision $w_\sigma \ast \mathcal{L}_\sigma$.

\item[Step 5: Why $\mathcal{L}_{p+1}$ is a simplicial complex]\hfill
\begin{enumerate}
\item \textbf{Closure under faces:} If $\Delta \in w_\sigma \ast \mathcal{L}_\sigma$, then any face of $\Delta$ is either a face of $\mathcal{L}_\sigma \subset \mathcal{L}_p$ or another cone face with apex $w_\sigma$; in both cases it lies in $\mathcal{L}_{p+1}$.
\item \textbf{Compatible overlaps:} If $\sigma,\sigma'$ are distinct $(p+1)$-simplices sharing a $p$-face $F$, then $\mathcal{L}_\sigma$ and $\mathcal{L}_{\sigma'}$ coincide on $F$ because both are restrictions of $\mathcal{L}_p$.
Consequently, the coned pieces $w_\sigma \ast \mathcal{L}_\sigma$ and $w_{\sigma'} \ast \mathcal{L}_{\sigma'}$ match along $F$, yielding no gaps or overlaps.
\end{enumerate}

\item[Step 6: Underlying space and subdivision property]
\begin{enumerate}
 \item On $\mathcal{K}^{(p)}$, the underlying space is unchanged and equals $|\mathcal{K}^{(p)}|$ (already subdivided by $\mathcal{L}_p$).
 \item Each $(p+1)$-simplex $\sigma$ is exactly filled by $w_\sigma \ast \mathcal{L}_\sigma$ with underlying space $|\sigma|$.
 \item Therefore $|\mathcal{L}_{p+1}| = |\mathcal{K}^{(p+1)}|$, and $\mathcal{L}_{p+1}$ refines $\mathcal{K}^{(p+1)}$ simplexwise; hence $\mathcal{L}_{p+1}$ is a subdivision of $\mathcal{K}^{(p+1)}$.
\end{enumerate}

\item[Step 7: Iteration (full sub-division of $\mathcal{K}$)] Starting from $\mathcal{L}_0 = \mathcal{K}^{(0)}$ (trivial subdivision on vertices) or from any chosen $\mathcal{L}_p$, repeat Steps 1–5 for $p=0,1,\dots,\dim \mathcal{K}-1$.
The result is a subdivision of $\mathcal{K}$ obtained by \emph{starring} interior points in increasing dimension. Special choices of $w_\sigma$ recover familiar constructions (e.g., barycentric subdivision when $w_\sigma$ is the barycenter of $\sigma$ for every $\sigma$).
\end{description}
\end{algorithm}

\section{Barycentric subdivision of complex}\label{s:Barycentric subdivision of complex}

Among various subdivision methods, barycentric subdivision is particularly important due to its systematic construction by using barycenters of faces and generating a refined simplicial complex. Subdivisions play a vital role in topology and computational geometry by enabling finer approximations, facilitating simplicial approximations of continuous maps, and supporting algorithms that require mesh refinement.

\begin{definition}[Barycenter of a simplex]\label{d:Barycenter of a simplex}
Let $\sigma=[v_0 v_1 \dots v_p]$ be an $n$-simplex in an affine (or Euclidean) space with vertices $v_0,\dots,v_p$.
The barycenter (or centroid) of $\sigma$ is the point
\[
b_{\sigma} = \frac{1}{p+1}\,\sum_{i=0}^{p} v_i,
\]
equivalently, the unique point whose barycentric coordinates relative to $\{v_i\}$ are all equal to $\tfrac{1}{p+1}$.
\end{definition}

\begin{example}[Barycenter of a 1-simplex]\label{eg:Barycenter of a 1-simplex}
Let $\sigma=[v_0 v_1]$ be a $1$-simplex in $\mathbb{R}^n$ with endpoints (vertices) $v_0$ and $v_1$.
The barycenter $\sigma=[v_0 v_1]$ is the midpoint
\[
b_{\sigma} = \frac{v_0 + v_1}{2}.
\]
Equivalently, in barycentric coordinates relative to the vertices $\{v_0, v_1\}$, the barycenter is the unique point with weights
\[
\lambda_{v_0} =\lambda_{v_1} =\tfrac{1}{2}, \qquad b_{\sigma} = \lambda_{v_0} v_0 + \lambda_{v_1} v_1.
\]
If $v_0=(a_1,\dots,a_n)$ and $v_0=(b_1,\dots,b_n)$, then
\[
b_{\sigma} = \Big(\tfrac{a_1+b_1}{2},\ \dots,\ \tfrac{a_n+b_n}{2}\Big).
\]

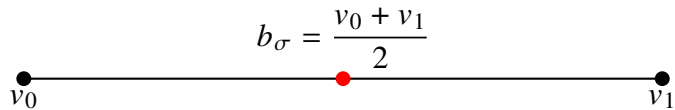
\begin{figure}[h!]
    \centering
    \begin{tikzpicture}[scale=2.8]

  \coordinate (A) at (0,0);
  \coordinate (B) at (3,0);

  \coordinate (M) at ($(A)!0.5!(B)$);

  \draw[thick] (A)--(B);

  \fill[red] (M) circle (0.035);

  \fill (A) circle (0.035);
  \fill (B) circle (0.035);

  \node[below] at (A) {$v_0$};
  \node[below] at (B) {$v_1$};
  \node[above] at (M) {$b_{\sigma}=\dfrac{v_0 + v_1}{2}$};


\end{tikzpicture}
    \caption{Barycenter of a 1-simplex}
    \label{fig:Barycenter of a 1-simplex}
\end{figure}
\end{example}

\begin{example}[Barycenter of a 2-simplex]\label{eg:Barycenter of a 2-simplex}
Let $\sigma = [v_0 v_1 v_2]$ be a $2$-simplex (triangle) in an affine/Euclidean space with vertices $\sigma = [v_0 v_1 v_2]$.
The barycenter (centroid) of $[v_0 v_1 v_2]$ is
\[
b_{\sigma} =  \frac{v_0 + v_1 + v_2}{3}.
\]
Equivalently, in barycentric coordinates relative to the vertices $\{v_0, v_1, v_2\}$,
\[
\lambda_{v_0}=\lambda_{v_1}=\lambda_{v_2}=\tfrac{1}{3},
\qquad
b_{\sigma} = \tfrac{1}{3}v_0 + \tfrac{1}{3}v_1 + \tfrac{1}{3}v_2.
\]
If $v_0=(x_0,y_0),\ v_1=(x_1,y_1),\ v_2=(x_2,y_2)\in\mathbb{R}^2$, then
\[
b_{\sigma} = \Big(\tfrac{x_0+x_1+x_2}{3},\ \tfrac{y_0+y_1+y_2}{3}\Big).
\]

\begin{figure}[h!]
    \centering
\begin{tikzpicture}[scale=4]

  \coordinate (v0) at (0,0);
  \coordinate (v1) at (1,0);
  \coordinate (v2) at (0.25,0.85);

  \coordinate (m0) at ($(v1)!0.5!(v2)$); 
  \coordinate (m1) at ($(v2)!0.5!(v0)$); 
  \coordinate (m2) at ($(v0)!0.5!(v1)$); 

  \coordinate (b) at ($ (v0)!.333333!(v1) !.333333! (v2) $);

  \fill[gray!10] (v0)--(v1)--(v2)--cycle;

  \draw[gray!70, dashed] (v0)--(m0);
  \draw[gray!70, dashed] (v1)--(m1);
  \draw[gray!70, dashed] (v2)--(m2);

  \draw[thick] (v0)--(v1)--(v2)--cycle;

  \fill (v0) circle (0.016);
  \fill (v1) circle (0.016);
  \fill (v2) circle (0.016);


  \node[below left]  at (v0) {$v_0$};
  \node[below right] at (v1) {$v_1$};
  \node[above]       at (v2) {$v_2$};
  \node[right]       at (b)  {$b_{\sigma}$};

\end{tikzpicture}
    \caption{Barycenter of a 2-simplex}
    \label{fig:Barycenter of a 2-simplex}
\end{figure}
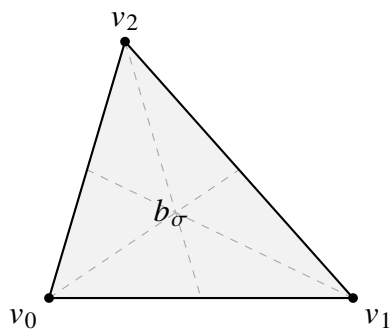
\end{example}

We now describe the general method for constructing complex sub-divisions  in term of barycentric division. 

\begin{definition}[$n$th barycentric subdivision of a complex] \label{d:nth barycentric subdivision of a complex}
Let $\mathcal{K}$ be a simplicial complex. Define a sequence $\{\mathcal{L}_p\}_{p\ge 0}$ of subdivisions of the skeleta of $\mathcal{K}$ inductively as follows.
Set $\mathcal{L}_0 := \mathcal{K}^{(0)}$, the $0$-skeleton of $\mathcal{K}$.
Assume $\mathcal{L}_p$ is a subdivision of the $p$-skeleton $\mathcal{K}^{(p)}$.
For each $(p+1)$-simplex $\sigma \in \mathcal{K}$, let $b_{\sigma}$ be the barycenter of $\sigma$, and let $\mathcal{L}_\sigma$ denote the restriction of $\mathcal{L}_p$ to $\operatorname{Bd}(\sigma)$.
Define
\[
\mathcal{L}_{p+1}
\;:=\;
\mathcal{L}_p \;\cup\;
\bigcup_{\substack{\sigma \in \mathcal{K}\\ \dim \sigma = p+1}}
\bigl( b_{\sigma} \ast \mathcal{L}_\sigma \bigr),
\]
i.e., $\mathcal{L}_{p+1}$ is obtained from $\mathcal{L}_p$ by starring $\mathcal{L}_p$ from the barycenters of all $(p+1)$-simplices of $\mathcal{K}$.
By Lemma~\eqref{l:15.2}, $\mathcal{L}_{p+1}$ is a complex and a subdivision of $\mathcal{K}^{(p+1)}$.

The \emph{first barycentric subdivision} of $\mathcal{K}$ is the union
\[
\operatorname{Bsd}(\mathcal{K})
\;:=\;
\bigcup_{p\ge 0} \mathcal{L}_p,
\]
which is a subdivision of $\mathcal{K}$.

Iterates are defined by
\[
\operatorname{Bsd}^{1}(\mathcal{K}) := \operatorname{Bsd}(\mathcal{K}), \quad
\operatorname{Bsd}^{2}(\mathcal{K}) := \operatorname{Bsd}(\operatorname{Bsd}(\mathcal{K})),
\quad
\operatorname{Bsd}^{\,n}(\mathcal{K}) := \operatorname{Bsd}\bigl(\operatorname{Bsd}^{\,n-1}(\mathcal{K})\bigr)
\quad (n\ge 2),
\]
and are called the $n$th barycentric subdivisions of $\mathcal{K}$.
\end{definition}


\begin{example}[Barycentric subdivisions of a simple complex]
Let $\mathcal{K}$ be the simplicial complex obtained by taking a triangle and attaching a line segment at one of its vertices.
Concretely, let the triangle be $[v_0 v_1 v_2]$ with all faces, and attach the edge $[v_0 w]$ at the vertex $v_0$.
Thus
\[
\mathcal{K}
\,=\,
\bigl\{
v_0,v_1,v_2,w;\ 
[v_0 v_1],[v_1 v_2],[v_2 v_0],[v_0 w];\
[v_0 v_1 v_2]
\bigr\}.
\]
\begin{description}
\item[Step 1: The $0$-skeleton and $\mathcal{L}_0$] By definition, $\mathcal{L}_0 = \mathcal{K}^{(0)} = \{v_0,v_1,v_2,w\}$.

\item[Step 2: The $1$-skeleton and $\mathcal{L}_1$ (first barycentric subdivision on $\mathcal{K}^{(1)}$)]
We subdivide each $1$-simplex in $\mathcal{K}^{(1)}$ by starring from its barycenter.
Introduce midpoints (barycenters of edges):
\[
m_{01} := \tfrac{v_0+v_1}{2},\quad
m_{12} := \tfrac{v_1+v_2}{2},\quad
m_{20} := \tfrac{v_2+v_0}{2},\quad
m_{0w} := \tfrac{v_0+w}{2}.
\]
Each original edge $[ab]$ is replaced by the two edges $[a\,m_{ab}]$ and $[m_{ab}\,b]$.
Therefore the $1$-skeleton subdivision is
\[
\mathcal{L}_1^{(1)} \,=\,
\bigl\{
[v_0 m_{01}], [m_{01} v_1],\ 
[v_1 m_{12}], [m_{12} v_2],\ 
[v_2 m_{20}], [m_{20} v_0],\ 
[v_0 m_{0w}], [m_{0w} w]
\bigr\},
\]
with vertex set $\{v_0,v_1,v_2,w, m_{01},m_{12},m_{20},m_{0w}\}$.
As there are no $2$-simplices yet used in this step, we set
\[
\mathcal{L}_1 \;=\; \mathcal{L}_0 \cup \mathcal{L}_1^{(1)}.
\]

\item[Step 3: The $2$-skeleton and $\mathcal{L}_2$ (complete first barycentric subdivision of $\mathcal{K}$)]
Now subdivide the unique $2$-simplex $[v_0 v_1 v_2]$ by starring from its barycenter
\[
b := \tfrac{v_0+v_1+v_2}{3}.
\]
Let $\mathcal{L}_{[v_0 v_1 v_2]}$ be the restriction of $\mathcal{L}_1$ to $\operatorname{Bd}[v_0 v_1 v_2] = [v_0 v_1]\cup[v_1 v_2]\cup[v_2 v_0]$.
On each boundary edge, $\mathcal{L}_1$ has already introduced the midpoint; coning from $b$ therefore yields the six $2$-simplices
\[
[b\,v_0\,m_{01}],\ [b\,m_{01}\,v_1],\ [b\,v_1\,m_{12}],\ [b\,m_{12}\,v_2],\ [b\,v_2\,m_{20}],\ [b\,m_{20}\,v_0].
\]
Collecting faces, the first barycentric subdivision is
\[
\operatorname{Bsd}(\mathcal{K}) \;=\; \mathcal{L}_2 \;=\; \mathcal{L}_1 \;\cup\; \bigl( b \ast \mathcal{L}_{[v_0 v_1 v_2]} \bigr).
\]
In particular, the edge $[v_0 w]$ remains subdivided only at $m_{0w}$, while the triangle is subdivided into six smaller triangles meeting at $b$.

\item[Step 4: The second barycentric subdivision $\operatorname{Bsd}^2(\mathcal{K})$]
Apply the same procedure to $\operatorname{Bsd}(\mathcal{K})$:
- On each edge of $\operatorname{Bsd}(\mathcal{K})$, insert its midpoint to subdivide it into two edges.
- On each $2$-simplex of $\operatorname{Bsd}(\mathcal{K})$ (there are six from the triangle region; the edge $[v_0 w]$ contributes no $2$-simplices), take its barycenter and cone from it over the subdivided boundary of that $2$-simplex.

Concretely, if $[x y z]$ is one of the six $2$-simplices in $\operatorname{Bsd}(\mathcal{K})$ with edge midpoints $m_{xy}, m_{yz}, m_{zx}$ and barycenter $b_{xyz} := \tfrac{x+y+z}{3}$, then $[x y z]$ is replaced by the six triangles
\[
[b_{xyz}\,x\,m_{xy}],\ [b_{xyz}\,m_{xy}\,y],\ [b_{xyz}\,y\,m_{yz}],\ [b_{xyz}\,m_{yz}\,z],\ [b_{xyz}\,z\,m_{zx}],\ [b_{xyz}\,m_{zx}\,x],
\]
together with all their faces; doing this for all six triangles yields the triangular region partitioned into $36$ small triangles.
Along the “tail” edge $[v_0 w]$, each of its two first-subdivision edges $[v_0 m_{0w}]$ and $[m_{0w} w]$ is further subdivided at its midpoint, giving four edges in total on that segment. Thus, $\operatorname{Bsd}^2(\mathcal{K}) \;=\; \operatorname{Bsd}\bigl(\operatorname{Bsd}(\mathcal{K})\bigr)$ is obtained by inserting midpoints on all edges of $\operatorname{Bsd}(\mathcal{K})$ and then coning from the barycenter of each $2$-simplex of $\operatorname{Bsd}(\mathcal{K})$ over its (already) subdivided boundary.
\end{description}

Thus, we can conclude that the $\operatorname{Bsd}(\mathcal{K})$ subdivides the attached edge $[v_0 w]$ at its midpoint and subdivides the triangle into six small triangles meeting at the barycenter $b$.
The $\operatorname{Bsd}^2(\mathcal{K})$ further subdivides each first-subdivision edge at its midpoint and refines each of the six triangles into six more, yielding $36$ triangles in the triangular region and four edges along the attached segment.


\begin{figure}[h!]
\centering

\begin{subfigure}[t]{0.47\textwidth}
\centering
\begin{tikzpicture}[scale=3]

  \coordinate (v0) at (0,0);
  \coordinate (v1) at (1,0);
  \coordinate (v2) at (0.2,0.85);
  \coordinate (w)  at (-0.55,-0.05);

  \fill[gray!10] (v0)--(v1)--(v2)--cycle;

  \draw[thick] (v0)--(v1)--(v2)--cycle;

  \draw[thick] (v0)--(w);

  \fill (v0) circle (0.016);
  \fill (v1) circle (0.016);
  \fill (v2) circle (0.016);
  \fill (w)  circle (0.016);

  \node[below]      at (v0) {$v_0$};
  \node[below]      at (v1) {$v_1$};
  \node[above]      at (v2) {$v_2$};
  \node[below left] at (w)  {$w$};

\end{tikzpicture}
\caption{Original complex $\mathcal{K}$}
\label{fig:K-original}
\end{subfigure}
\hfill
\begin{subfigure}[t]{0.47\textwidth}
\centering
\begin{tikzpicture}[scale=3]

  \coordinate (v0) at (0,0);
  \coordinate (v1) at (1,0);
  \coordinate (v2) at (0.2,0.85);
  \coordinate (w)  at (-0.55,-0.05);

  \coordinate (m01) at ($(v0)!0.5!(v1)$);
  \coordinate (m12) at ($(v1)!0.5!(v2)$);
  \coordinate (m20) at ($(v2)!0.5!(v0)$);
  \coordinate (m0w) at ($(v0)!0.5!(w)$);

  \coordinate (b) at ($ (v0)!1/3!(v1) !1/3! (v2) $);

  \fill[blue!10]  (b)--(v0)--(m01)--cycle;
  \fill[blue!10]  (b)--(m01)--(v1)--cycle;
  \fill[red!10]   (b)--(v1)--(m12)--cycle;
  \fill[red!10]   (b)--(m12)--(v2)--cycle;
  \fill[green!10] (b)--(v2)--(m20)--cycle;
  \fill[green!10] (b)--(m20)--(v0)--cycle;

  \draw[thick] (v0)--(v1)--(v2)--cycle;

  \draw[thick, dashed] (b)--(v0);
  \draw[thick, dashed] (b)--(v1);
  \draw[thick, dashed] (b)--(v2);
  \draw[thick, dashed] (b)--(m01);
  \draw[thick, dashed] (b)--(m12);
  \draw[thick, dashed] (b)--(m20);

  \draw[thick] (v0)--(m0w)--(w);

  \fill (v0) circle (0.016);
  \fill (v1) circle (0.016);
  \fill (v2) circle (0.016);
  \fill (w)  circle (0.016);
  \fill (m01) circle (0.013);
  \fill (m12) circle (0.013);
  \fill (m20) circle (0.013);
  \fill (m0w) circle (0.013);
  \fill[red] (b) circle (0.015);

  \node[below]      at (v0) {$v_0$};
  \node[below]      at (v1) {$v_1$};
  \node[above]      at (v2) {$v_2$};
  \node[below left] at (w)  {$w$};
  \node[below]      at (m01) {\scriptsize $m_{01}$};
  \node[right]      at (m12) {\scriptsize $m_{12}$};
  \node[left]       at (m20) {\scriptsize $m_{20}$};
  \node[below]       at (m0w) {\scriptsize $m_{0w}$};
  \node[right]      at (b)   {\scriptsize $b$};

\end{tikzpicture}
\caption{First barycentric subdivision $\operatorname{Bsd}(\mathcal{K})$}
\label{fig:K-bsd1}
\end{subfigure}

\begin{subfigure}[t]{0.98\textwidth}
\centering
\begin{tikzpicture}[scale=4, line cap=round, line join=round]

  \coordinate (v0) at (0,0);
  \coordinate (v1) at (1,0);
  \coordinate (v2) at (0.22,0.84);
  \coordinate (w)  at (-0.62,-0.06);

  \coordinate (m01) at ($(v0)!0.5!(v1)$);
  \coordinate (m12) at ($(v1)!0.5!(v2)$);
  \coordinate (m20) at ($(v2)!0.5!(v0)$);
  \coordinate (b)   at ($ (v0)!1/3!(v1) !1/3! (v2) $);

  \coordinate (m0w) at ($(v0)!0.5!(w)$);   
  \coordinate (s1)  at ($(v0)!0.5!(m0w)$); 
  \coordinate (s2)  at ($(m0w)!0.5!(w)$);  

  \definecolor{cA}{RGB}{239,83,80}
  \definecolor{cB}{RGB}{255,167,38}
  \definecolor{cC}{RGB}{255,238,88}
  \definecolor{cD}{RGB}{102,187,106}
  \definecolor{cE}{RGB}{38,198,218}
  \definecolor{cF}{RGB}{66,165,245}
  \definecolor{cG}{RGB}{171,71,188}
  \definecolor{cH}{RGB}{255,112,67}
  \definecolor{cI}{RGB}{156,204,101}
  \definecolor{cJ}{RGB}{77,208,225}
  \definecolor{cK}{RGB}{144,202,249}
  \definecolor{cL}{RGB}{186,104,200}

  \newcount\coloridx
  \coloridx=0\relax
  \def\nextfill{%
    \ifnum\coloridx=0 \def\thiscol{cA}\fi
    \ifnum\coloridx=1 \def\thiscol{cB}\fi
    \ifnum\coloridx=2 \def\thiscol{cC}\fi
    \ifnum\coloridx=3 \def\thiscol{cD}\fi
    \ifnum\coloridx=4 \def\thiscol{cE}\fi
    \ifnum\coloridx=5 \def\thiscol{cF}\fi
    \ifnum\coloridx=6 \def\thiscol{cG}\fi
    \ifnum\coloridx=7 \def\thiscol{cH}\fi
    \ifnum\coloridx=8 \def\thiscol{cI}\fi
    \ifnum\coloridx=9 \def\thiscol{cJ}\fi
    \ifnum\coloridx=10 \def\thiscol{cK}\fi
    \ifnum\coloridx=11 \def\thiscol{cL}\fi
    \advance\coloridx by 1\relax
    \ifnum\coloridx>11 \coloridx=0\relax\fi
  }

  \newcommand{\SecondBsdColored}[3]{%
    \coordinate (pq) at ($ (#1)!0.5!(#2) $);
    \coordinate (qr) at ($ (#2)!0.5!(#3) $);
    \coordinate (rp) at ($ (#3)!0.5!(#1) $);
    \coordinate (cT) at ($ (#1)!1/3!(#2) !1/3! (#3) $);

    \nextfill \fill[\thiscol!45] (cT)--(#1)--(pq)--cycle;
    \nextfill \fill[\thiscol!45] (cT)--(pq)--(#2)--cycle;
    \nextfill \fill[\thiscol!45] (cT)--(#2)--(qr)--cycle;
    \nextfill \fill[\thiscol!45] (cT)--(qr)--(#3)--cycle;
    \nextfill \fill[\thiscol!45] (cT)--(#3)--(rp)--cycle;
    \nextfill \fill[\thiscol!45] (cT)--(rp)--(#1)--cycle;

    \draw[thin, black!70] (#1)--(pq)--(#2)--(qr)--(#3)--(rp)--(#1);
    \draw[thin, black!60] (cT)--(#1) (cT)--(#2) (cT)--(#3) (cT)--(pq) (cT)--(qr) (cT)--(rp);
  }

  \SecondBsdColored{b}{v0}{m01}
  \SecondBsdColored{b}{m01}{v1}
  \SecondBsdColored{b}{v1}{m12}
  \SecondBsdColored{b}{m12}{v2}
  \SecondBsdColored{b}{v2}{m20}
  \SecondBsdColored{b}{m20}{v0}

  \draw[thick] (v0)--(v1)--(v2)--cycle;
  \draw[thick] (v0)--(m01)--(v1);
  \draw[thick] (v1)--(m12)--(v2);
  \draw[thick] (v2)--(m20)--(v0);

  \draw[thin, dashed, black!50] (b)--(v0) (b)--(v1) (b)--(v2) (b)--(m01) (b)--(m12) (b)--(m20);

  \draw[very thick] (v0)--(s1)--(m0w)--(s2)--(w);

  \fill (v0) circle (0.018);
  \fill (v1) circle (0.018);
  \fill (v2) circle (0.018);
  \fill (w)  circle (0.018);

  \fill (m01) circle (0.013);
  \fill (m12) circle (0.013);
  \fill (m20) circle (0.013);
  \fill[red] (b) circle (0.017);

  \fill (m0w) circle (0.013);
  \fill (s1) circle (0.011);
  \fill (s2) circle (0.011);

  \node[below=2pt] at (v0) {$v_0$};
  \node[below=2pt] at (v1) {$v_1$};
  \node[above=2pt] at (v2) {$v_2$};
  \node[below left=0pt] at (w)  {$w$};
  \node[right=2pt] at (b)  {\scriptsize $b$};
  \node[below=2pt] at (m0w) {\scriptsize $m_{0w}$};

\end{tikzpicture}
\caption{Second barycentric subdivision $\operatorname{Bsd}^2(\mathcal{K})$}
\label{fig:K-bsd2}
\end{subfigure}

\caption{First and second barycentric sub-division of complex $\mathcal{K}$}
\label{fig:First and second barycentric sub-division of complex K}
\end{figure}
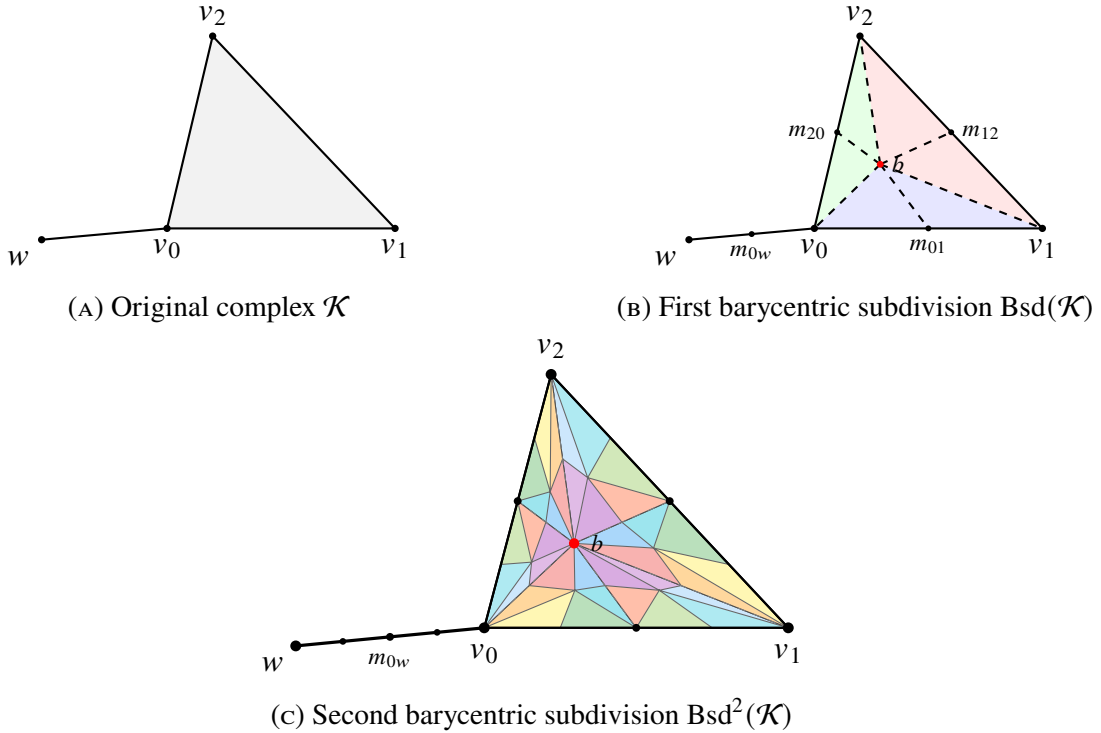
\end{example}

\begin{lemma}[Barcentric subdivision via flags of faces]\label{l:Barcentric subdivision via flags of faces}
The complex $\operatorname{Sd}(\mathcal{K})$ equals the collection of all simplices of the form 
$b_{\sigma_{1}} b_{\sigma_{2}} \cdots b_{\sigma_{2}}$, where $\sigma_{1} > \sigma_{2} \cdots > \sigma_{n}$, and $\sigma_{1} > \sigma_{2}$ implies $\sigma_{2}$ is proper face of $\sigma_{1}$ and $b_{\sigma_{i}}$ is the barycenter of $\sigma = [v_{0}v_{1}\cdots v_{p}]$ is defined by the point 
\[b_{\sigma_{i}} = \sum_{i = 0}^{p} \dfrac{1}{p + 1}v_{i}\]
\end{lemma}

The objective of next theorem is to guarantee that iterated barycentric subdivision makes every simplex uniformly small in the given metric, so the mesh of the triangulation can be made finer than any prescribed scale $\epsilon$.

\begin{theorem}[Mesh goes to zero under iterated barycentric subdivision]
Given a finite simplicial complex $\mathcal{K}$, a metric on its realization $|\mathcal{K}|$, and $\epsilon>0$, there exists $N\in\mathbb{N}$ such that every simplex $\tau$ of $\operatorname{Bsd}^{N}(\mathcal{K})$ has diameter $\operatorname{diam}(\tau)<\epsilon$.
\end{theorem}

\begin{proof}
As given  a finite simplicial complex $\mathcal{K}$, a metric $d$ on the geometric realization $|\mathcal{K}|$ and tolerance $\epsilon>0$. We want to find an integer $N$ so that in the $N$-fold barycentric subdivision $\operatorname{Bsd}^{N}(\mathcal{K})$, every simplex has $d$-diameter less than $\epsilon$.

Let $\mathrm{mesh}(\mathcal{L}) := \max\{\operatorname{diam}(\sigma): \sigma\in \mathcal{L}\}$ denote the mesh (maximal simplex diameter) of a geometric simplicial complex $\mathcal{L}$ with respect to the given metric $d$.
Because $\mathcal{K}$ is finite, $\mathrm{mesh}(\mathcal{K})<\infty$.

It suffices to prove that there exists a constant $c\in(0,1)$, depending only on the dimension $n:=\dim\mathcal{K}$, such that
\[
\mathrm{mesh}\bigl(\operatorname{Bsd}(\Delta)\bigr) \;\le\; c\,\mathrm{mesh}(\Delta)
\quad\text{for every simplex }\Delta\text{ of dimension }\le n.
\]
Granting this, an induction over the simplices of $\mathcal{K}$ yields
\[
\mathrm{mesh}\bigl(\operatorname{Bsd}(\mathcal{K})\bigr) \;\le\; c\,\mathrm{mesh}(\mathcal{K}),
\]
and iterating,
\[
\mathrm{mesh}\bigl(\operatorname{Bsd}^{m}(\mathcal{K})\bigr)
\;\le\; c^{\,m}\,\mathrm{mesh}(\mathcal{K}) \;\xrightarrow[m\to\infty]{}\; 0.
\]
Given $\epsilon>0$, choose $N$ with $c^{N}\mathrm{mesh}(\mathcal{K})<\epsilon$; then every simplex $\tau$ of $\operatorname{Bsd}^{N}(\mathcal{K})$ satisfies $\operatorname{diam}(\tau)\le \mathrm{mesh}(\operatorname{Bsd}^{N}(\mathcal{K}))<\epsilon$.

It remains to justify the contraction factor $c\in(0,1)$.
For an $n$-simplex $\Delta$ in an affine/Euclidean model, each simplex $\Delta'$ appearing in $\operatorname{Bsd}(\Delta)$ is a convex hull of barycenters of a flag of faces of $\Delta$.
A standard estimate shows that
\[
\operatorname{diam}(\Delta') \;\le\; \frac{n}{n+1}\,\operatorname{diam}(\Delta),
\]
since every vertex of $\Delta'$ is an average of vertices of $\Delta$ with coefficients summing to $1$, and successive averaging contracts distances by at least the factor $n/(n+1)$.
Thus one may take $c = n/(n+1) < 1$ uniformly for all $n$-simplices, and the same bound applies to lower-dimensional faces (with even stronger contraction).
This proves the claim.
\end{proof}

\section{Conclusion}\label{s:Conclusion}
In this paper, we have thoroughly explored the theory and methodology of subdivisions of simplicial complexes, focusing particularly on barycentric subdivisions. We have presented clear definitions, detailed examples, and rigorous proofs to demonstrate how iterated barycentric subdivision refines the complex, making all simplices arbitrarily small with respect to any given metric. This property is crucial for applications in topology and geometry such as simplicial approximation and mesh refinement in computational geometry. The illustrative diagrams provided shed light on the stepwise process of subdivision and help visualize the intricate structure of the resulting complexes. Overall, our exposition aims to provide both theoretical insights and practical tools for researchers working with simplicial complexes and their refinements.

For readers interested in exploring the theory and applications of subdivisions of simplicial complexes in greater depth, we recommend consulting foundational texts such as Munkres’ \emph{Elements of Algebraic Topology} \cite{munkres1984elements} and Spanier’s \emph{Algebraic Topology} \cite{spanier1981algebraic}. For a modern perspective on subdivisions in the context of polytopes and higher-dimensional complexes, Ziegler’s \emph{Lectures on Polytopes} \cite{Ziegler2012} is an excellent resource. In 1988, Margaret \cite{margaret1988} characterized simplicial polytopes that are barycentric subdivisions of the regular $CW$ sphere. In 2010, Murai studied  face vectors of barycentric subdivisions of simplicial homology manifolds \cite{murai2010}. In 2021, Athanasiadis discussed about the face number of barycentric subdivisions of cubical complexes in \cite{Athanasiadis2021}. These references offer extensive theoretical background, examples, and applications that complement and deepen the material presented in this paper.

\bibliographystyle{unsrtnat}
\bibliography{MainBibFile}
\end{document}